\documentclass [aap,preprint]{imsart}
\arxiv{math.PR/1508.05162}
\RequirePackage[OT1]{fontenc}
\RequirePackage{amsthm,amsmath}
\usepackage[pdftex]{graphicx}
\RequirePackage[colorlinks,citecolor=blue,urlcolor=blue]{hyperref}

\startlocaldefs
\newtheorem{thm}{Theorem}[section]
\newtheorem{lemma}[thm]{Lemma}
\newtheorem{prop}[thm]{Proposition}

\newcommand{\Exp}{\text{$\bf E$}}

\newcommand{\sgn}{\text{sgn}}
\newcommand{\reals}{{\sf I \! R}}
\newcommand{\complexs}{{\sf I \! \! C}}
\newcommand{\zbar}{\bar{z}}
\newcommand{\ds}{\displaystyle}
\newcommand{\ba}{\begin{array}}
\newcommand{\ea}{\end{array}}
\renewcommand{\Re}{\text{\rm Re}}
\renewcommand{\Im}{\text{\rm Im}}
\endlocaldefs

\begin{document}

\begin{frontmatter}

\title{The complex zeros of random sums}
\runtitle{Zeros of Random Sums}

\begin{aug}
\author{\fnms{Robert J.}
	\snm{Vanderbei}\corref{}\ead[label=e1]{rvdb@princeton.edu}\thanksref{t1}\ead[label=e1]{rvdb@princeton.edu}}
\thankstext{t1}{Research supported by ONR through grant	N00014-13-1-0093 and N00014-16-1-2162}
\runauthor{Robert J. Vanderbei}
\affiliation{Princeton University}
\address{Dept. of Ops. Res. and Fin. Eng. \\
	Princeton University \\ 
	Princeton, NJ 08544 
}

\vspace*{0.3in}
{\em This paper is dedicated to the memory of Larry Shepp. }
\vspace*{0.3in}

\end{aug}

\begin{abstract}
    This paper extends earlier work on the distribution in the complex plane of
    the roots of random polynomials.  In this paper, the random polynomials are
    generalized to random finite sums of given ``basis'' functions.  The basis
    functions are assumed to be entire functions that are real-valued on the
    real line.  The
    coefficients are assumed to be independent identically distributed Normal
    $(0,1)$ random variables.  An explicit formula for the density function is
    given in terms of the set of basis functions.  We also consider some practical
    examples including Fourier series.  In some cases, we derive an explicit
    formula for the limiting density as the number of terms in the sum tends to
    infinity.
\end{abstract}

\begin{keyword}[class=MSC]
	\kwd[Primary ]{30C15}
	\kwd[; secondary ]{30B20, 26C10, 60B99}
\end{keyword}
\end{frontmatter}


\section{Introduction.}

	The problem of characterizing the distribution of the roots of random
	polynomials has a long history.  In 1943, Kac \cite{Kac43} studied the
	real roots of random polynomials with iid normal coefficients.  He
	obtained an explicit formula for the density function for the
	distribution of the real roots.

    Following the initial work of Kac, 
    a large body of research on
    zeros of random polynomials has appeared -- see \cite{BRS86} for a fairly
    complete account of the early work in this area
    including an extensive list of references.  
    Most of this early work focused on the real zeros; \cite{ET50}, 
    \cite{Ham56} and \cite{SS62} being a few notable exceptions.  
    The paper of Edelman and Kostlan 
    \cite{EK95} gives a very elegent geometric treatment of the problem. 

    In more recent years, the work has branched off in a number of directions.
    For example, in 1995, Larry Shepp and I derived an explicit formula for the
    distribution of the roots in the complex plane (see \cite{Van94d}) when the
    coefficients are assumed to be iid normal random variables. 
    A short time later, Ibragimov and Zeitouni \cite{IZ97} took a different
    approach and were able to rederive our results and also find limiting
    distributions as the degree $n$ tends to infinity under more general
    distributional assumptions.  See also \cite{KZ13} and \cite{KZ14}.
    
    Also in the late 1990's, it was pointed out that
    understanding deeper statistical properties of the random roots, such as
    $k$-point correlations among the roots, was both interesting mathematically
    and had important implications in physics (see, e.g., \cite{Pro96},
		    \cite{FH99} and \cite{SM09}).
    
    A number of papers have appeared that attempt to prove certain specific
    properties under increasingly general distributional assumptions.  For example, in
    2002, Dembo et al. \cite{DPSZ02} derived a formula for the probability that
    none of the roots are real (assuming $n$ is even, of course) in the case
    when the coefficients of the polynomial are iid but not necessarily normal.
    Other papers have continued to study real roots---see, e.g., \cite{Wil97}.
    Another property that has been actively studied in recent years is the fact that
    as $n$ gets large the complex roots tend to distribute themselves close to
    and uniformly about the unit circle in the complex planes--see, e.g., the
    papers by Shiffman and Zelditch \cite{SZ03}, Hughes and Nikeghbali
    \cite{HN08}, Ibragimov and Zaporozhets \cite{IZ13}, Pritsker and Yeager
    \cite{PY15} and Pritsker \cite{Pri17}.
    Also, Li and Wei \cite{LW09} have considered harmonic
    polynomials--polynomials in the complex variable $z$ and it's conjugate
    $\bar{z}$.

    Using a very different approach, Feldheim \cite{Fel12} has derived a result
    that with some work can be shown to be equivalent to the results presented
    herein.

    Recently, Tao and Vu \cite{TV14}, drawing on the close connection with
    random matrix theory, derived asymptotic formulas for the correlation 
    functions of the roots of random polynomials.  They specifically address the
    question of how many zeros are real.

    The results summarized above mostly establish certain
    properties of the roots under very general distributional assumptions.  The
    price paid for that generality is that most results only hold
    asymptotically as $n \rightarrow \infty$.  
    In contrast, this
    paper introduces a modest generalization to the core assumptions underlying
    the results in \cite{Van94d} and we show that analogous explicit formulas
    can still be derived for any value of $n$.
    Specifically, instead of considering polynomials, 
    $\sum_{j=0}^n \eta_j z^j$, we generalize the ``basis'' functions $z^j$ to be
    any set of entire functions, $f_j(z)$, that are real on the real line.  So, to that end, we
	let
	\[
		P_n(z) = \sum_{j=0}^{n} \eta_j f_j(z), \qquad z \in \complexs,
	\]
	where $n$ is a fixed integer, the $\eta_j$'s are independent identically
	distributed $N(0,1)$ random variables, and the functions $f_j$ are given
	entire functions that are real-valued on the real line.
	We derive an explicit formula for the
	expected number of zeros in any measurable subset $\Omega$ of the 
	complex plane $\complexs$.  
	The formula will be expressed in terms of the following functions:
		\[
		    \begin{array}{rclrclrcl}
			A_0 (z) &=& \ds \sum_{j=0}^{n} f_j(z)^{2}, & \quad
			B_0 (z) &=& \ds \sum_{j=0}^{n} |f_j(z)|^{2}, & \quad
			\\[3ex]
			A_1 (z) &=& \ds \sum_{j=0}^{n} f_j(z) f_j'(z), &
			B_1 (z) &=& \ds \sum_{j=0}^{n} \overline{f_j(z)} f_j'(z), 
			\\[3ex]
			A_2 (z) &=& \ds \sum_{j=0}^{n} f_j'(z)^{2}, & \quad
			B_2 (z) &=& \ds \sum_{j=0}^{n} |f_j'(z)|^2, 
		    \end{array}
		\]
		and
		\begin{equation} \label{24}
			D_0 (z) = \sqrt{ B_0(z)^2 - |A_0(z)|^2}.
		\end{equation}
		and, lastly, 
		\[
		    E_1(z) = \sqrt{A_2(z) A_0(z) - A_1(z)^2}.
		\]
	Notes: (1) As usual, an overbar denotes complex conjugation and 
	 primes denote differentiation with respect to $z$.
	 (2) The function $E_1$ will only be needed in places where the argument of
	 the square root is a positive real.  At such places, 
	 the square root is assumed to be a positive real.
	 (3) Throughout the paper, we follow the usual convention of denoting 
	 the real and imaginary parts
	 of a complex variable $z$ by $x$ and $y$, respectively, i.e., $z = x + iy$. 

	Let $\nu_n(\Omega)$ denote the (random) number of zeros of $P_n$ in a
	set $\Omega$ in the complex plane.  Our first theorem asserts that
	throughout most of the plane this random variable has a density with
	respect to Lebesgue measure on the plane:
	\begin{thm} \label{thm1}
		For each measurable set $\Omega \subset \{z \in \complexs \; | \; D_0(z) \ne 0 \}$,
		\begin{equation} \label{21}
			\Exp \nu_n(\Omega) = \int_\Omega h_n(x,y) dxdy ,
		\end{equation}
		where
		\[
			h_n(z)    = \frac{
				   B_2 D_0^2 -
				   B_0 (|B_1|^2 + |A_1|^2) +
				   (A_0 B_1 \overline{A}_1 + \overline{A_0 B_1} A_1)
				 }{
				   \pi D_0^3
				 } .
		\]
	\end{thm}
	It is easy to see that the density function $h_n$ is real
	valued.  It is less obvious that it is nonnegative.  We leave this
	sanity check to the reader.
	As we see from the above theorem, places where $D_0$ vanishes are special and
	must be studied separately.   The real axis is one such place:
	\begin{thm} \label{thm2}
		On the real line, the function $D_0$ vanishes.  For each
		measurable set $\Omega \subset \reals$,
		\begin{equation} \label{21}
			\Exp \nu_n(\Omega) = \int_{\Omega} g_n(x) dx ,
		\end{equation}
		where
		\[
			g_n(x) = \frac{
				   E_1(x)
				 }{
				   \pi B_0
				 } .
		\]
	\end{thm}

	In the case where the $f_j(z)$'s are just powers of $z$, our results
	reduce to those given in \cite{Van94d}.
	The proof here
	parallels the analogous proof given in \cite{Van94d} but there are a few
	differences, the most important one being the fact that, in
	general, the function $B_1$ is not real-valued like it was in
	\cite{Van94d}.  It
	seems that this explicit formula has not been derived before
	and, as shown in later sections, there are interesting
	new examples that can now be solved.

	While the definition of $h_n$ in Theorem \ref{thm1} looks rather
	complicated, it is nevertheless amenable both to computation 
	and, with some choices of the
	functions $f_j$, it is amenable to asymptotic analysis as well.  

    In the following section, we derive the explicit formulas given above
    for the intensity 
    functions $h_n$ and $g_n$.  
    Then, in Section
    \ref{sec5}, we look at some specific examples and finally, in Section
    \ref{sec6}, we offer some speculation and suggest future research
    directions.

\section{The Intensity Functions $h_n$ and $g_n$.} \label{sec2}

    This section is devoted to the proof of Theorems \ref{thm1} and \ref{thm2}.  
    We begin with the following proposition.

    \begin{prop} \label{prop1}
	For each region $\Omega \in \complexs$ whose boundary intersects the
	set 
	$\{z \; | \; D_0(z) = 0 \}$ 
	at most only finitely many times,
	\begin{equation} \label{26}
		\Exp \nu_n(\Omega) = \frac{1}{2 \pi i}
				\int_{\partial \Omega} F(z) dz ,
	\end{equation}
	where 
	\begin{equation} \label{27}
		F = \frac{
		      B_1 D_0 + B_0 B_1 - \bar{A}_0 A_1
		    }{
		      B_0 D_0 + B_0^2   - \bar{A}_0 A_0
		    } .
	\end{equation}
    \end{prop}

    \begin{proof}
	The argument principle (see, e.g., \cite{Ahl66}, p. 151) gives an
	explicit formula for the random variable $\nu_n(\Omega)$, namely
	\begin{equation} \label{7}
		\nu_n(\Omega) = \frac{1}{2 \pi i} 
			   \int_{\partial \Omega} \frac{P_n'(z)}{P_n(z)} dz.
	\end{equation}

	Taking expectations in \eqref{7} and then
	interchanging expectation and contour integration (the justification
	of which is tedious but doable), we get
	\begin{equation} \label{25}
		\Exp \nu_n(\Omega) = \frac{1}{2 \pi i} 
			            \int_{\partial \Omega} 
				    \Exp \frac{P_n'(z)}{P_n(z)} 
				    dz .
	\end{equation}
	The following Lemma shows that, away from the set
	$\{z \; | \; D_0(z) = 0 \}$, the function
	\begin{equation} \label{30}
		F(z) = \Exp \frac{P_n'(z)}{P_n(z)}
	\end{equation}
	simplifies to the expression given in \eqref{27} and, since we've assumed
	that $\partial \Omega$ intersects this set at only finitely many
	points, this finishes the proof.
    \end{proof}

    \begin{lemma} \label{lemma1}
        Let $F$ denote the function defined by \eqref{30}.
	For $z \not\in \{z \; | \; D_0(z) = 0 \}$,
        \[
	    F = \frac{
	          B_1 D_0 + B_0 B_1 - \bar{A}_0 A_1
	        }{
	          B_0 D_0 + B_0^2   - \bar{A}_0 A_0
	        } .
        \]
    \end{lemma}

    \begin{proof}
	Note that
	$P_n(z)$ and $P_n'(z)$ are complex Gaussian random variables.  It is
	convenient to work with their real and imaginary parts,
	\begin{eqnarray} 
		P_n(z)  & = & \xi_1 + i \xi_2 ,	\nonumber \\
		P_n'(z) & = & \xi_3 + i \xi_4 ,	\nonumber 
	\end{eqnarray} 
	which are just linear combinations of the original standard normal
	random variables:
	\[
	    \ba{rclcrcl}
		\xi_1 & = & \ds \sum_{j=0}^{n} a_j \eta_j , & \quad
		\xi_2 & = & \ds \sum_{j=0}^{n} b_j \eta_j , \\
		\xi_3 & = & \ds \sum_{j=0}^{n} c_j \eta_j , & \quad
		\xi_4 & = & \ds \sum_{j=0}^{n} d_j \eta_j .
	    \ea
	\]
	The coefficients in these linear combinations are given by
	\begin{equation} \label{11}
	    \ba{rclll}
		a_j & = & \Re(f_j(z)) & = & \dfrac{f_j(z) + \overline{f_j(z)}}{2} , \\[0.1in]
		b_j & = & \Im(f_j(z)) & = & \dfrac{f_j(z) - \overline{f_j(z)}}{2i} , \\[0.1in]
		c_j & = & \Re(f_j'(z)) & = & \dfrac{f_j'(z) + \overline{f_j'(z)}}{2} , \\[0.1in]
		d_j & = & \Im(f_j'(z)) & = & \dfrac{f_j'(z) - \overline{f_j'(z)}}{2i} .
	    \ea
	\end{equation} 
	Put
	$
		\xi = [
			  \ba{cccc}
			      \xi_1  &
			      \xi_2  &
			      \xi_3  &
			      \xi_4
			  \ea
		      ]^T .
	$
	The covariance among these four Gaussian random variables is easy to
	compute:
	\begin{equation} \label{13}
		\text{Cov}(\xi) = \Exp \xi \xi^T
				= \left[
				      \ba{cccc}
					  a^T a & a^T b & a^T c & a^T d \\
					  b^T a & b^T b & b^T c & b^T d \\
					  c^T a & c^T b & c^T c & c^T d \\
					  d^T a & d^T b & d^T c & d^T d 
				      \ea
				  \right]
	\end{equation}
	We now represent these four correlated Gaussian
	random variables in terms of four independent standard normals.  
	To this end, we seek a lower triangular matrix 
	$L = [ \ba{c} \l_{ij} \ea ]$ 
	such that the vector $\xi$ is equal in
	distribution to $L \zeta$, where 
	$
		\zeta = [
			  \ba{cccc}
			      \zeta_1 &
			      \zeta_2 &
			      \zeta_3 &
			      \zeta_4
			  \ea
		        ]^T 
	$
	is a vector of four independent standard normal random variables.  
	The following simple
	calculation shows that $L$ is the Cholesky factor for the
	covariance matrix:
	\begin{equation} \label{14}
		\text{Cov}(\xi) = \Exp \xi \xi^T = \Exp L \zeta \zeta^T L^T
				= L L^T .
	\end{equation} 
	Now, since $\xi \stackrel{\text{D}}{=} L \zeta$ and $L$ is
	lower triangular
	(the symbol $\stackrel{\text{D}}{=}$ denotes equality in
	distribution), we get that
	\begin{eqnarray*}
		\frac{P_n'(z)}{P_n(z)} 
			& = & \frac{\xi_3 + i \xi_4}{\xi_1 + i \xi_2} \\
			& \stackrel{\text{D}}{=} & 
			      \frac{
				  (l_{31} + i l_{41}) \zeta_1 +
				  (l_{32} + i l_{42}) \zeta_2 +
				  (l_{33} + i l_{43}) \zeta_3 +
				            i l_{44}  \zeta_4 
			      }{
				  (l_{11} + i l_{21}) \zeta_1 +
				            i l_{22}  \zeta_2 
			      } .
	\end{eqnarray*}
	Hence, exploiting the independence of the $\zeta_i$'s, we see that
	\begin{equation} \label{12}
		F(z) 
		=
		\Exp
		\frac{P_n'(z)}{P_n(z)} 
		=
		\Exp
		      \frac{
			  \alpha \zeta_1 + \beta \zeta_2
		      }{
			  \gamma \zeta_1 + \delta \zeta_2
		      } ,
	\end{equation} 
	where
	\[
	    \ba{rclcrcl}
		\alpha & = & l_{31} + i l_{41} & \quad &
		\beta  & = & l_{32} + i l_{42} \\
		\gamma & = & l_{11} + i l_{21} & \quad &
		\delta & = &          i l_{22} .
	    \ea
	\]
	Splitting up the numerator in \eqref{12} and exploiting the
	exchangeability of $\zeta_1$ and $\zeta_2$, we can rewrite the 
	expectation as follows:
	\[
		F(z)
		=
		\frac{\alpha}{\delta} \; f ( \gamma / \delta ) +
		\frac{\beta }{\gamma} \; f ( \delta / \gamma ) ,
	\]
	where $f$ is a complex-valued function defined on 
	$\complexs \setminus \reals$ by
	\[
		f(w) = \Exp \frac{\zeta_1}{w \zeta_1 + \zeta_2} .
	\]
	The expectation appearing in the definition of $f$ can be explicitly
	computed.  Indeed,
	\begin{eqnarray*}
		f(w) & = & \frac{1}{2 \pi}
		           \int_0^{2 \pi} \int_0^{\infty}
		           \frac{ 
			       \rho \cos \theta 
		           }{
			       w \rho \cos \theta + \rho \sin \theta
		           }
		           e^{ - \rho^2 /2 } \rho d \rho d \theta \\
		     & = & \frac{1}{2 \pi}
			   \int_0^{2 \pi}
			   \frac{d \theta}{w + \tan \theta} 
	\end{eqnarray*}
	and this last integral can be computed explicitly giving us
	\[
		f(w) = \left\{
			   \ba{ll}
				\dfrac{1}{w+i}, & \quad \Im(w) > 0,\\[2ex]
				\dfrac{1}{w-i}, & \quad \Im(w) < 0. 
			   \ea
		       \right.
	\]
	Recalling the definition of $\delta$ and $\gamma$, we see that
	\[
		\frac{\gamma}{\delta} 
			= \frac{l_{21}}{l_{22}} - i \frac{l_{11}}{l_{22}} .
	\]
	In general, $l_{11}$ and $l_{22}$ are just nonnegative.  However, it 
	is not
	hard to show that they are both strictly positive whenever $z$ has a
	nonzero imaginary part.  Hence, $\gamma / \delta$
	lies in the lower half-plane,
	$\delta / \gamma$ lies in the upper half-plane, and 
	\begin{eqnarray}
		F(z)
		& = &
		\frac{\alpha}{\delta}
		\frac{1}{\frac{\gamma}{\delta}-i}
		+
		\frac{\beta}{\gamma}
		\frac{1}{\frac{\delta}{\gamma}+i}          \label{15} \\
		& = &
		\frac{i \alpha + \beta}{i \gamma + \delta} \nonumber \\
		& = &
		\frac{
			l_{32}-l_{41} + i (l_{31}+l_{42})
		}{
			-l_{21} + i (l_{11}+l_{22})
		} .					   \nonumber
	\end{eqnarray}
	At this point, we need explicit formulas for the elements of the
	Cholesky factor $L$.  From \eqref{13} and \eqref{14}, we see that
	\[
	    \begin{array}{rclrcl}
		a^T a & = & l_{11}^2      \\[2ex]
		b^T a & = & l_{21} l_{11} & \qquad \qquad
		b^T b & = & l_{21}^2 + l_{22}^2           \\[2ex]
		c^T a & = & l_{31} l_{11} &
		c^T b & = & l_{31} l_{21} + l_{32} l_{22} \\[2ex]
		d^T a & = & l_{41} l_{11} &
		d^T b & = & l_{41} l_{21} + l_{42} l_{22} .
	    \end{array}
	\]
	Solving these equations in succession, we get
	\[
	    \begin{array}{rclrcl}
		l_{11} & = & \dfrac{a^T a}{\sqrt{a^T a}} \\[0.2in]
		l_{21} & = & \dfrac{b^T a}{\sqrt{a^T a}} & \qquad \qquad
		l_{22} & = & \dfrac{
				(a^T a)(b^T b) - (b^T a)^2
			     }{
				\sqrt{a^T a} R
			     } \\[0.2in]
		l_{31} & = & \dfrac{c^T a}{\sqrt{a^T a}}  &
		l_{32} & = & \dfrac{
				(a^T a)(c^T b) - (c^T a)(b^T a)
			     }{
				\sqrt{a^T a} R
			     } \\[0.2in]
		l_{41} & = & \dfrac{d^T a}{\sqrt{a^T a}}  &
		l_{42} & = & \dfrac{
				(a^T a)(d^T b) - (d^T a)(b^T a)
			     }{
				\sqrt{a^T a} R
			     } 
	    \end{array}
	\]
	where
	\[
		R = \sqrt{ (a^T a)(b^T b) - (b^T a)^2 } .
	\]
	Substituting these expressions into \eqref{15} and simplifying, 
	we see that
	\begin{equation} \label{16}
		F(z)
		= 
		\frac{
		    -d^T a + ic^T a - i
		    \left( 
			a^T a (-d^T b + i c^T b) - (-d^T a + i c^T a) b^T a 
		    \right) 
		    / R
		}{
		    -b^T a + i a^T a + i R
		} .
	\end{equation} 
	Recalling the definitions of $a_j$, $b_j$, $c_j$, and $d_j$ given in
	\eqref{11}, it is easy to check that the following identities hold:
	\[
	    \begin{array}{rclrcl}
		a^T a & = & \phantom{+}\frac{1}{4} ( A_0 + 2 B_0 + \bar{A}_0 ) , \\[2ex]
		b^T a & = & -\frac{i}{4} ( A_0 - \bar{A}_0 )         , & \qquad 
		b^T b & = & -\frac{1}{4} ( A_0 - 2 B_0 + \bar{A}_0 ) , \\[2ex]
		c^T a & = & \phantom{+}\frac{1}{4} ( A_1 + B_1 + \bar{B}_1 + \bar{A}_1 ) , &
		c^T b & = & -\frac{i}{4} ( A_1 - B_1 + \bar{B}_1 - \bar{A}_1 )         , \\[2ex]
		d^T a & = & -\frac{i}{4} ( A_1 + B_1 - \bar{B}_1 - \bar{A}_1 )         , &
		d^T b & = & -\frac{1}{4} ( A_1 - B_1 - \bar{B}_1 + \bar{A}_1 ) . 
	    \end{array}
	\]
	Plugging these expressions into \eqref{16} and simplifying, we get that
	\begin{equation} \label{17}
		F(z)
		= 
		\frac{
		    A_1 + B_1 +
		    ( A_0 B_1 + B_0 B_1 - A_1 B_0 - \bar{A}_0 A_1 )
		    / D_0
		}{
		    A_0 + B_0 + D_0
		} ,
	\end{equation} 
	where $D_0$ is as given in \eqref{24}.
	It turns out that further simplification occurs if we make the
	denominator real by the usual technique of
	multiplying and dividing by its complex conjugate.  We leave out the
	algebraic details except to mention that a factor of 
	$A_0 + 2B_0 + \bar{A}_0$ cancels out from the numerator and
	denominator leaving us with
	\begin{equation} \label{19}
		F(z)
		= 
		\frac{
		    B_1 D_0 + B_0 B_1 - \bar{A}_0 A_1
		}{
		    D_0 ( B_0 + D_0 )
		} ,
	\end{equation} 
	or, expanding out $D_0^2$,
	\begin{equation} \label{20}
		F(z)
		= 
		\frac{
		    B_1 D_0 + B_0 B_1 - \bar{A}_0 A_1
		}{
		    B_0 D_0 + B_0^2   - \bar{A}_0 A_0
		} .
	\end{equation} 
    \end{proof}

    \begin{lemma} \label{lemma1b}
        On the real axis, $F$ has a jump discontinuity.  
	Indeed, for each $a \in \reals$,
        \[
	    \lim_{z \rightarrow a: ~ \Im(z) > 0} F
	    =
	    \frac{B_1(a) - i \; E_1(a)}{B_0(a)}
        \]
        and
        \[
	    \lim_{z \rightarrow a: ~ \Im(z) < 0} F
	    =
	    \frac{B_1(a) + i \; E_1(a)}{B_0(a)} .
        \]
    \end{lemma}

    \begin{proof}
	Consider a point $a$ on the real axis.  On the reals, $A_k = B_k$,
	for $k=0,1$, and so $D_0 = 0$.  Hence, the right-hand side in
	\eqref{19} is an indeterminate form.  To analyze the limiting behavior
	of $F$ near the real axis, we first divide the numerator and
	denominator by $D_0$:
	\begin{equation} \label{31}
		F = \frac{
		      B_1 + \dfrac{B_0 B_1 - \bar{A}_0 A_1}{D_0}
		    }{
		      B_0 + D_0
		    } .
	\end{equation}
	Now, only the ratio in the numerator is indeterminate.
	To study it, we start by expressing things in terms of the $f_j$
	functions:
	\begin{eqnarray*}
	    B_0 B_1 - \bar{A}_0 A_1 
	    & = &
	    \sum_{j,k} \overline{f_j(z)} f_k'(z) 
	    \left( f_j(z) \overline{f_k(z)} - \overline{f_j(z)} f_k(z) \right) \label{100} \\
	    & = &
	    2 i \sum_{j,k} \overline{f_j(z)} f_k'(z) 
	    \Im\left( f_j(z) \overline{f_k(z)} \right) \nonumber 
	\end{eqnarray*}
	and
	\begin{eqnarray*}
	    D_0^2 \;\; = \;\; B_0^2 - |A_0|^2
	    & = &
	    \sum_{j,k} f_j(z) \overline{f_k(z)}
	    \left( \overline{f_j(z)} f_k(z) - f_j(z) \overline{f_k(z)} \right) \label{101} \\
	    & = &
	    -2 i \sum_{j,k} f_j(z) \overline{f_k(z)}
	    \Im\left( f_j(z) \overline{f_k(z)} \right). \nonumber 
	\end{eqnarray*}
	Next, we write the first few terms of the Taylor series expansion of the $f_j$'s
	about the point $z=a$,
	substitute the expansions into the formulas above and then drop ``high'' order
	terms to derive the first few terms of the Taylor expansions for
	$B_0 B_1 - \bar{A}_0 A_1$ and $B_0^2 - |A_0|^2$.
	For the first expression, we only need to go to linear terms to get
	\begin{eqnarray*}
	    B_0 B_1 - \bar{A}_0 A_1 
	    & = &
	    2 i \sum_{j,k} f_j(a) f_k'(a) 
	        \left( f_j'(a) f_k(a) - f_j(a) f_k'(a) \right) \; y \; 
		+ \; o(z-a) \\
	    & = &
	    2 i \left( A_1(a)^2  - A_0(a) A_2(a) \right) y \; + \; o(z-a) 
        \end{eqnarray*}
	(as usual, we use $y$ to denote the imaginary part of $z$).
	For the second expression, we need to go to quadratic terms.  The
	result is
	\begin{eqnarray*}
	    D_0^2 
	    & = &
	    4 \sum_{j,k} \left( 
			      f_j'(a)^2 f_k(a)^2 - f_j(a) f_j'(a) f_k(a) f_k'(a)
			    \right) y^2 \; + \; o((z-a)^2) \\
	    & = &
	    4 \left( 
		   \left(\sum_{j=0}^n f_j'(a)^2 \right) 
		   \left(\sum_{j=0}^n f_j(a)^2 \right) -
		   \left(\sum_{j=0}^n f_j'(a) f_j(a) \right)^2 
	      \right) y^2 \; + \; o((z-a)^2) \\[2ex]
	    & = &
	    4 \left( A_2(a) A_0(a) - A_1(a)^2 \right) y^2 \; + \; o((z-a)^2) .
	\end{eqnarray*}
	Hence, we see that
	\begin{equation}
	    \dfrac{B_0 B_1 - \bar{A}_0 A_1}{D_0}
	    = 
	    -i \; E_1(a) \; \sgn(a) \; \sgn(y)  \; + \; o(z-a) .
	    \label{32}
	\end{equation}
	Combining \eqref{31} and \eqref{32}, we get the desired limits expressing
	the jump discontinuity on the real axis.
    \end{proof}

    \subsubsection*{Proof of Theorem \ref{thm1}.}
	Without loss of generality, it suffices to consider regions $\Omega$
	that are either regions that do not intersect the real axis or small
	rectangles centered on the real axis.
	We begin by considering a region $\Omega$ that does not intersect the 
	real axis.  Applying Stokes' theorem to the expression for 
	$\Exp \nu_n (\Omega)$ given in Proposition \ref{prop1}, we see that
	\[
		\Exp \nu_n(\Omega) = \frac{1}{\pi} \int_{\Omega}
				\frac{\partial}{\partial \bar{z}}
				F(z, \bar{z})
				dx dy .
	\]
	Note that we are now writing $F(z, \bar{z})$ to emphasize the fact
	that $F$ depends on both $z$ and $\bar{z}$.
	Letting the dagger symbol stand for the derivative with respect to
	$\bar{z}$, we see from Lemma \ref{lemma1} that
	\begin{eqnarray} \label{200} 
		~ \\
		\frac{\partial F}{\partial \bar{z}}
		& = &
		\left\{
		    (B_0 D_0 + B_0^2 - \bar{A}_0 A_0)
		    (B_1^{\dagger} D_0 + B_1 D_0^{\dagger} +B_0^{\dagger} B_1
		     + B_0 B_1^{\dagger} - \bar{A}_0^{\dagger} A_1)
		\right.  \nonumber \\
		& &
		\left.
		    - (B_1 D_0 + B_0 B_1 - \bar{A}_0 A_1)
		    (B_0^{\dagger} D_0 + B_0 D_0^{\dagger} + 2 B_0
		     B_0^{\dagger} - \bar{A}_0^{\dagger} A_0)
		\right\} \nonumber \\
		& &
		/ (B_0 D_0 + B_0^2 - |A_0|^2)^2 .  \nonumber
	\end{eqnarray}
	Recall that we have assumed that the functions $f_j$ are entire and
	are real-valued on the real line.   Hence, they have the property that
	$\overline{f_j(z)} = f_j(\bar{z})$.  Their derivatives also have this
	property.  Exploiting these facts,
	it is easy to check that
	\begin{equation} \label{201}
		B_0^{\dagger} = \bar{B}_1, \quad
		\bar{A}_0^{\dagger} = 2 \bar{A}_1, \quad
		B_1^{\dagger} = B_2.
	\end{equation}
	Recalling that $D_0 = \sqrt{B_0^2 - |A_0|^2}$, we get that
	\begin{equation} \label{202}
		D_0^{\dagger} = \frac{B_0 \bar{B}_1 - A_0 \bar{A}_1}{D_0} .
	\end{equation}
	As explained in \cite{Van94d},
	substituting these formulas for the derivatives into 
	the expression given above for $\partial F/ \partial \bar{z}$
	followed by
	careful algebraic simplifications (see the appendix for the details)
	eventually leads to the fact that
	$(1/\pi) \partial F(z,\bar{z}) / \partial \bar{z}$ equals
	the expression given for $h_n$ in the statement of the theorem.
     \qed

    \subsubsection*{Proof of Theorem \ref{thm2}.}
	Consider a narrow rectangle that straddles an interval of the
	real axis:
	$\Omega = [a,b]\times[-\varepsilon,\varepsilon]$ where $a < b$ and
	$\varepsilon > 0$.
	Writing the contour integral for $\Exp \nu_n(\Omega)$ given by
	Proposition \ref{prop1} and letting $\varepsilon$ tend to $0$, we see that
	\[
	    \Exp \nu_n ( (a,b) )
	    =
	    \frac{1}{2 \pi i} \int_{a}^{b} \left( F(x-)-F(x+) \right) dx ,
	\]
	where $\nu_n ( (a,b) )$ denotes the number of zeros in the
	interval $(a,b)$ of the real axis and
	\[
	    F(x-) \; = \; \lim_{z \rightarrow x: \Im(z) < 0} F(z)
	    \quad
	    \text{ and }
	    \quad
	    F(x+) \; = \; \lim_{z \rightarrow x: \Im(z) > 0} F(z) .  
	\]
	From Lemma \ref{lemma1b}, we see that
	\[
	    g_n(x) 
	    = 
	    \frac{1}{2 \pi i} (F(x-)-F(x+))
	    = \frac{ E_1(x) }{\pi B_0} .
	\]
	This completes the proof.
     \qed

\section{Examples.} \label{sec5}

In this section, we consider some examples.
The simplest example corresponds to the $f_j$ simply being the power functions:
\[
    f_j(z) = z^j .
\]
As this case was studied carefully in \cite{Van94d},
other that showing a particular example ($n = 10$) in Figure \ref{fig1}, 
we refer the reader to that previous paper for more information about this
example.

    \begin{figure}[t]
    \begin{center}
    \includegraphics[width=5.0in]{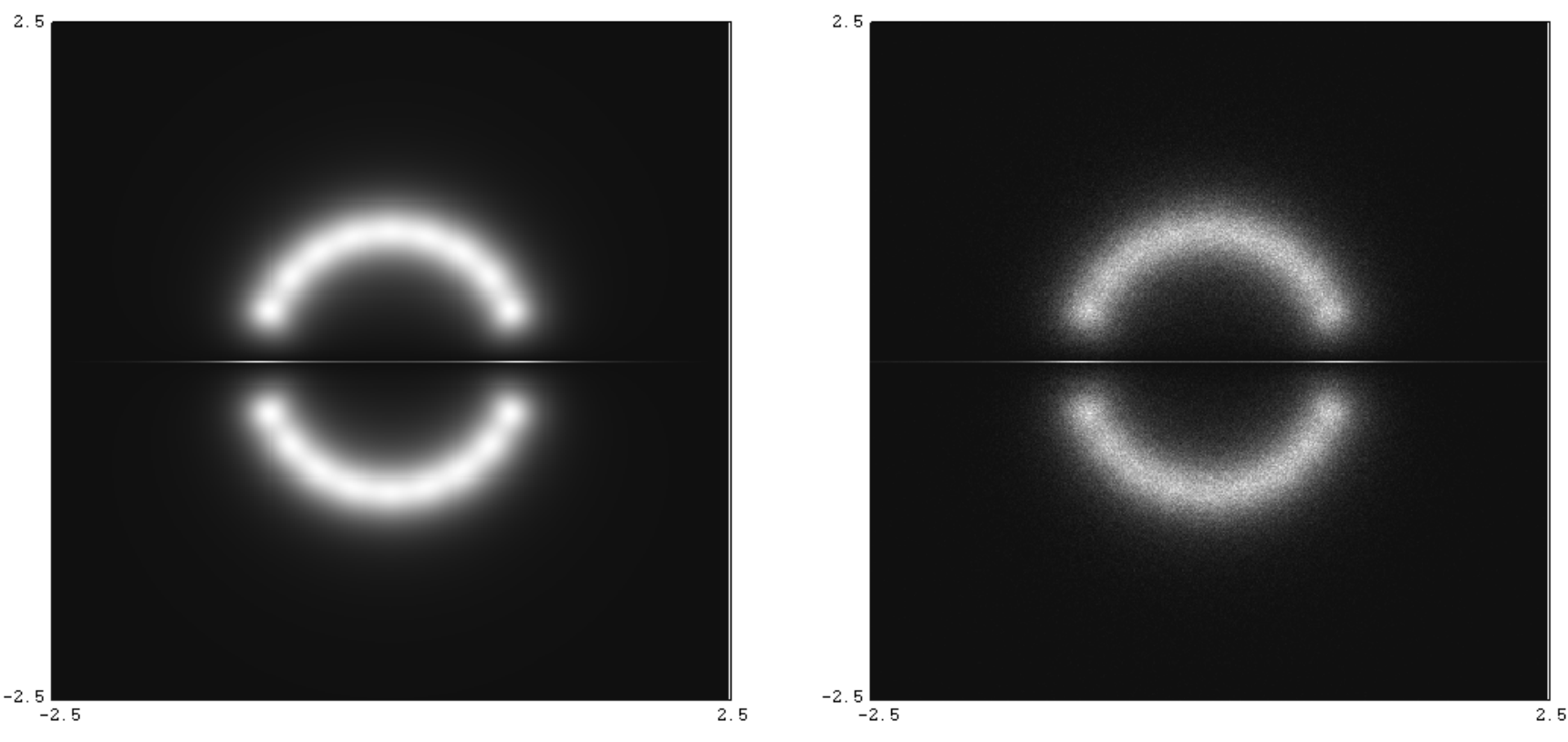}
    \end{center}
    \caption{Random degree $10$ polynomial: 
	    $\eta_0 + \eta_1 z + \eta_2 z^2 + \cdots + \eta_{10} z^{10}$.
    In this figure and the following ones,
    the left-hand plot is a grey-scale image of the intensity functions $h_n$
    and $g_n$ (the latter being concentrated on the x-axis).
    The right-hand plot shows $200{,}000$ roots from randomly
    generated polynomials.
    Note that, for the left-hand plots,
    the grey-scales for $h_n$ and $g_n$ are scaled
    separately and in such a way that
    both use the full range from white to black.
    } 
    \label{fig1}
    \end{figure}

    Each figure in this section shows two plots.  
    On the left is a grey-scale plot of the intensity
    functions $g_n$ and $h_n$.  On the right is a plot of hundreds of thousands of
    zeros obtained by
    generating random sums and explicitly finding their zeros.
    
    The intensity plots appearing on the left were produced by partitioning the 
    given square domain into a $440$ by $440$ grid of ``pixels'' 
    and computing the intensity function in the 
    center of each pixel.  The grey-scale was computed by assigning black 
    to the pixel with the smallest value and white to the pixel with the 
    largest value and then linearly interpolating all values in between.
    This grey-scale computation was performed separately for $h_n$ and for $g_n$
    (which appears only on the x-axis) and so no conclusions should be drawn
    comparing the intensity shown on the x-axis with that shown off from it.
    The applet used to produce these figures can be found at
    \begin{center}
        \href{http://www.princeton.edu/~rvdb/JAVA/Roots/Roots.html}{http://www.princeton.edu/$\sim$rvdb/JAVA/Roots/Roots.html}
    \end{center}
    
    Of course, the intensity function $g_n$ is one-dimensional and
    therefore it would be natural (and more informative) to make separate 
    plots of values of $g_n$ verses $x$, but such plots appear in many 
    places (see, e.g., \cite{Kac59}) and so it seemed unnecessary to produce 
    them here.  
    
\subsection{Weyl Polynomials}

Sums in which the $f_j$'s are given by 
\[
    f_j(z) = \frac{z^j}{\sqrt{j!}}
\]
are called {\em Weyl polynomials} (also sometimes called {\em flat polynomials}).
Figure \ref{fig2} shows the empirical distribution for the case where $n=10$.
For this case, the limiting forms of the various functions defining the
densities are easy to compute:
\[
    \begin{array}{rclrclrcl}
        \ds \lim_{n \rightarrow \infty} A_0(z) & = & e^{z^2} & \qquad
        \ds \lim_{n \rightarrow \infty} B_0(z) & = & e^{|z|^2} \\[1ex]
        \ds \lim_{n \rightarrow \infty} A_1(z) & = & z e^{z^2} & \qquad
        \ds \lim_{n \rightarrow \infty} B_1(z) & = & \zbar e^{|z|^2} \\[1ex]
        \ds \lim_{n \rightarrow \infty} A_2(z) & = & (z^2 + 1) e^{z^2} & \qquad
        \ds \lim_{n \rightarrow \infty} B_2(z) & = & (|z|^2 + 1) e^{|z|^2} \\[3ex]
        \ds \lim_{n \rightarrow \infty} D_0(z) & = & \sqrt{e^{2(x^2+y^2)} - e^{4(x^2-y^2)}} & \qquad
        \ds \lim_{n \rightarrow \infty} E_1(z) & = & e^{z^2} .
    \end{array}
\]
The random Weyl polynomials are interesting because in the limit as $n
\rightarrow \infty$, the distribution of the real roots becomes uniform over the
real line:
\begin{thm} If $f_j(z) = z^j/\sqrt{j!}$ for all $j$, then
\[
    \lim_{n \rightarrow \infty} g_n(x) = \frac{1}{\pi}.
\]
\end{thm}
\begin{proof}
Follows trivially from Theorem \ref{thm2} and the formulas above.
\end{proof}

    \begin{figure}[t]
    \begin{center}
    \includegraphics[width=5.0in]{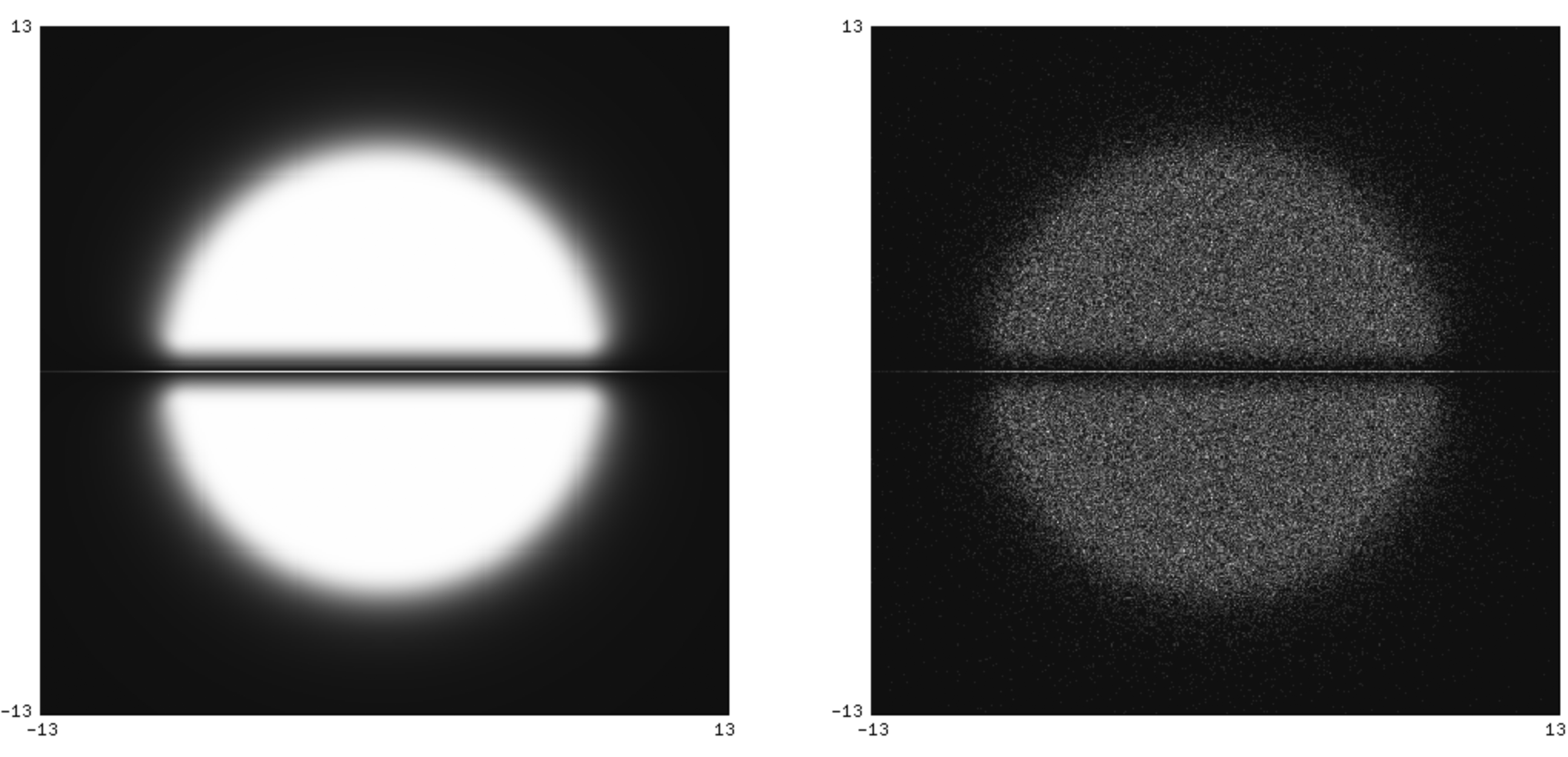}
    \end{center}
    \caption{Random degree $10$ Weyl polynomials:
	    $\eta_0 + \eta_1 z + \eta_2 \frac{z^2}{\sqrt{2!}} 
	    + \cdots + \eta_{80} \frac{z^{80}}{\sqrt{80!}}$.
	    The empirical distribution on the right was generated using
	    $2{,}500$ random sums.
    }
    \label{fig3}
    \end{figure}

It is interesting to note that, in addition to the asymptotic uniformity of the
distribution of the real roots, the complex roots are also much more uniformly
distributed than was the case when we did not have the $1/\sqrt{j!}$ factors.

Figure \ref{fig10} shows plots of $g_n$ for all of the examples considered here. 
    \begin{figure}[t]
    \begin{center}
    \includegraphics[width=5.0in]{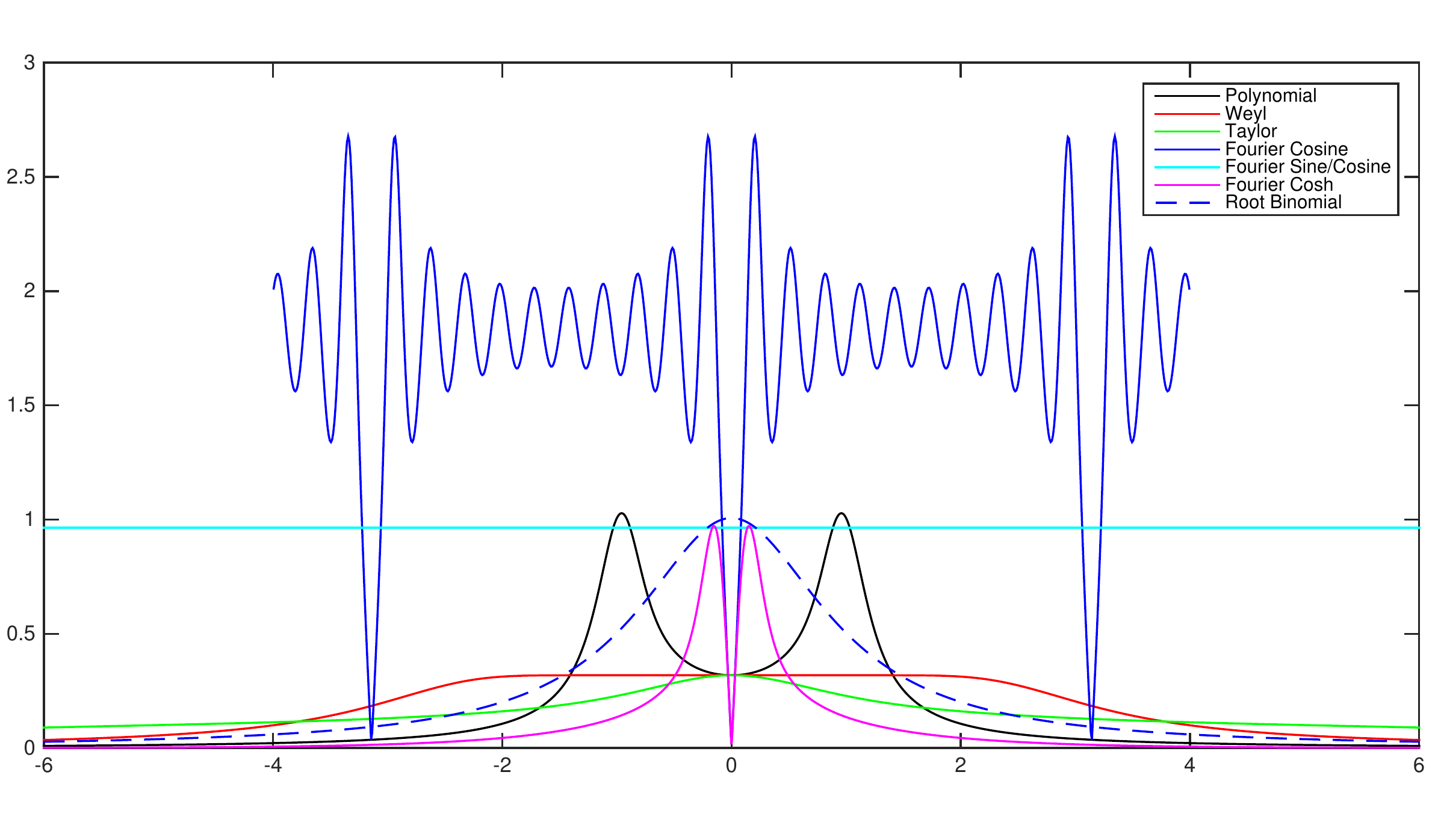}
    \end{center}
    \caption{The function $g_n$ for $n=10$ for several choices of the $f_j$'s
    }
    \label{fig10}
    \end{figure}

\subsection{Taylor Polynomials}

Another obvious set of polynomials to consider are the random Taylor polynomials; i.e.,
those polynomials with 
\[
    f_j(z) = \frac{z^j}{j!} .
\]
Figure \ref{fig3} shows the $n=10$ empirical distribution for these polynomials.

    \begin{figure}[t]
    \begin{center}
    \includegraphics[width=5.0in]{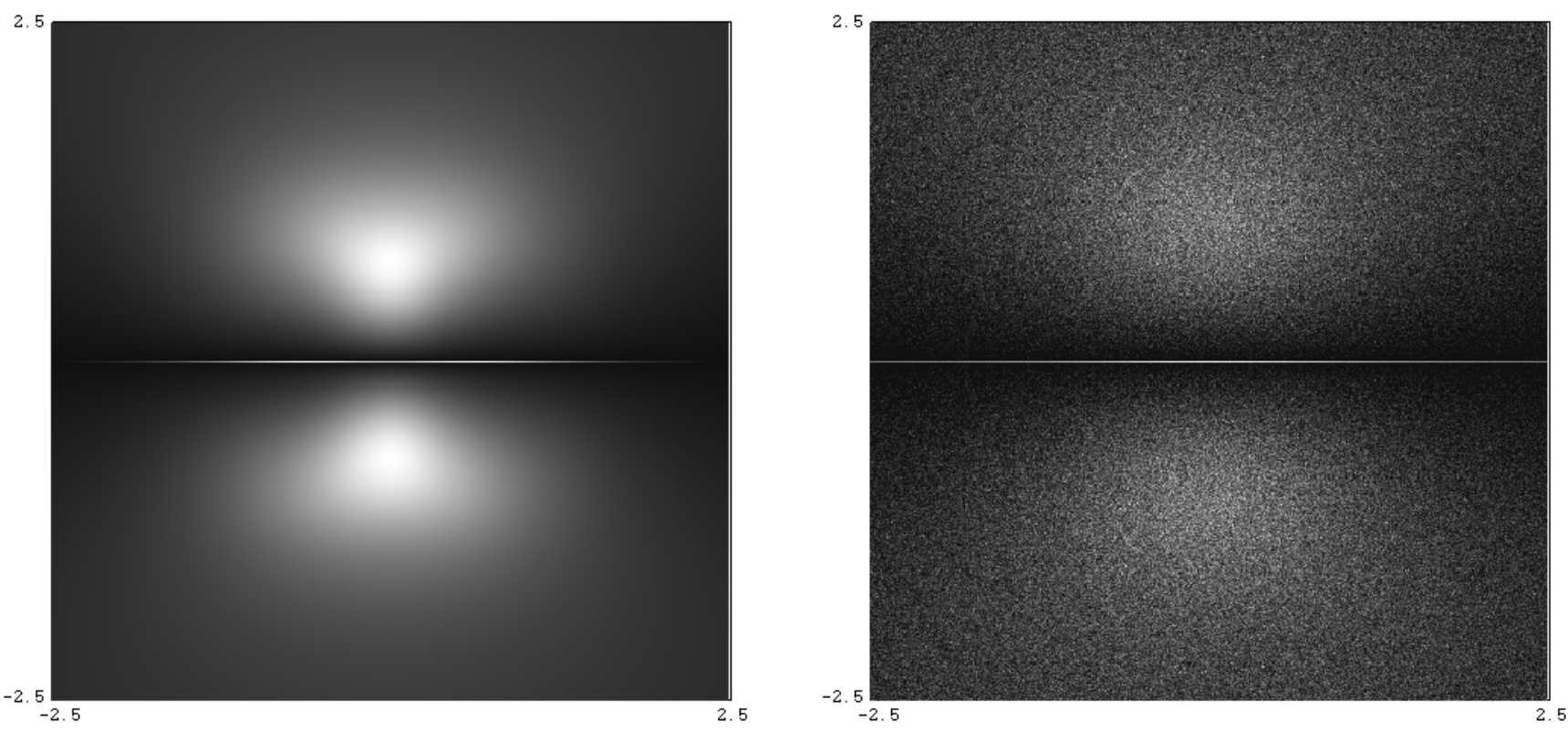}
    \end{center}
    \caption{Random degree $10$ Taylor polynomials:  
	    $\eta_0 + \eta_1 z + \eta_2 \frac{z^2}{2!} 
	    + \cdots + \eta_{10} \frac{z^{10}}{10!}$.
	    The empirical distribution on the right was generated using
	    $500{,}000$ random sums.
    } \label{fig2}
    \end{figure}

\subsection{Root-Binomial Polynomials}

Let 
\[
    f_j(z) = \sqrt{{n}\choose{j}} \; z^j .
\]
Figure \ref{fig22} shows the $n=10$ empirical distribution for these polynomials.
This example is interesting because the real and complex density functions take on a rather
simple explicit form.
Indeed, it is easy to check that
\[
    \begin{array}{rclrclrcl}
	A_0 (z) &=& (1+z^2)^n ,
	& \quad
	B_0 (z) &=& (1+|z|^2)^n ,
	\\[3ex]
	A_1 (z) &=& n z (1+z^2)^{n-1} ,
	& \quad
	B_1 (z) &=& n \zbar (1+|z|^2)^{n-1} ,
	\\[3ex]
	A_2 (z) &=& n (1+nz^2) (1+z^2)^{n-2} ,
	& \quad
	B_2 (z) &=& n (1+n|z|^2) (1+|z|^2)^{n-2} .
    \end{array}
\]
The formula for the density on the real axis simplifies nicely:
\[
    g(x) = \frac{\sqrt{n}}{\pi} \frac{1}{1+x^2} .
\]
From this formula, we see that the expected number of real roots is $\sqrt{n}$
and that each real root has a Cauchy distribution.

    \begin{figure}[t]
    \begin{center}
    \includegraphics[width=5.0in]{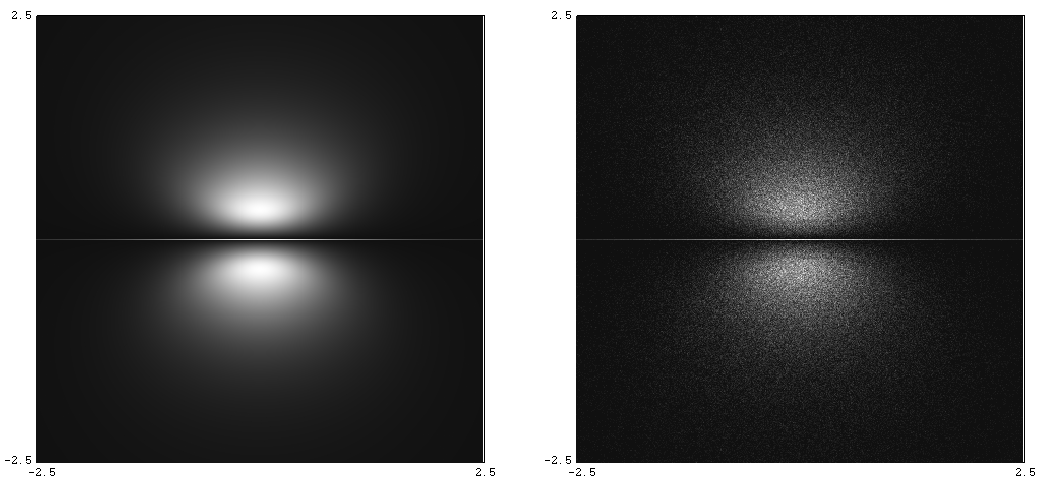}
    \end{center}
    \caption{Random degree $10$ root-binomial polynomials:  
	    $\eta_0 
	    + \eta_1 \sqrt{\frac{10}{1}} z 
	    + \eta_2 \sqrt{\frac{10 \cdot 9}{2 \cdot 1}} z^2
	    + \eta_3 \sqrt{\frac{10 \cdot 9 \cdot 8}{3 \cdot 2 \cdot 1}} z^3
	    + \cdots 
	    + \eta_{10} z^{10}$.
	    The empirical distribution on the right was generated using
	    $50{,}000$ random sums.
    } \label{fig22}
    \end{figure}

\subsection{Fourier Cosine Series}

Now let's consider a family of random sums that are not polynomials, namely, 
random (truncated) Fourier cosine series:
\[
    f_j(z) = \cos(j z).
\]
This case is interesting because these functions are real-valued not only on the
real axis but on the imaginary axis as well: $\cos(iy) = \cosh(y)$.  
Hence, $D_0$ vanishes on both the
real and the imaginary axes and, therefore, both axes have a density of zeros.
The set of imaginary roots for a particular sum using the $f_j$'s map to a set
of real roots if $f_j(z)$ is replaced with $\tilde{f}_j(z) = f_j(iz) =
\cosh(z)$.  Hence, the formula for the density on the imaginary axis is easy 
to compute by this simple rotation.
The resulting density on the imaginary axis has this simple form:
\[
    g_n(y) = \frac{E_1(iy)}{\pi B_0(iy)} .
\]
An example with $n=10$ is shown in Figure \ref{fig4}.

    \begin{figure}[t]
    \begin{center}
    \includegraphics[width=5.0in]{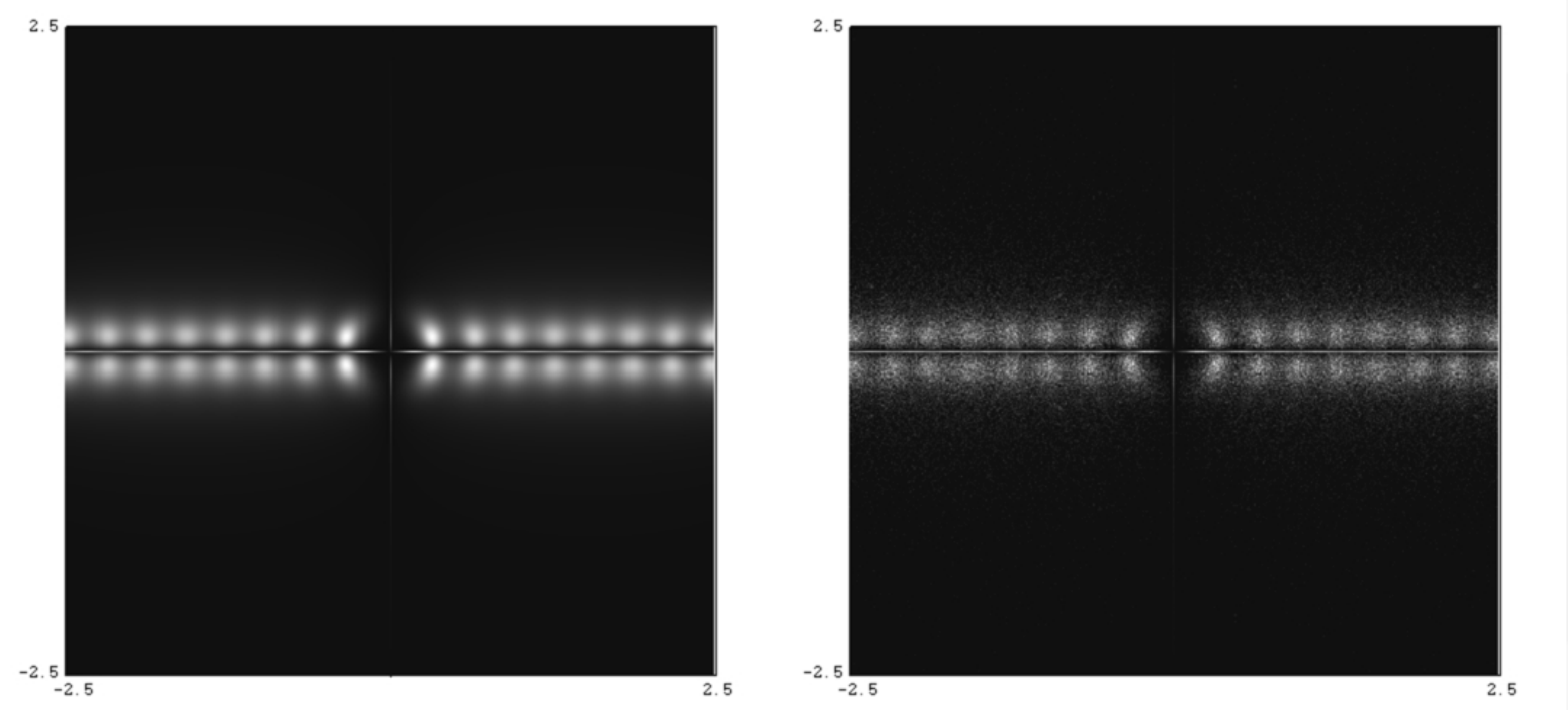}
    \end{center}
    \caption{Random sum of first $11$ terms in a cosine Fourier series:
	    $\eta_0 + \eta_1 \cos(z) + \eta_2 \cos(2z)
	    + \cdots + \eta_{10} \cos(10 z)$.
	    The empirical distribution on the right was generated using
	    $24{,}000$ random sums.
    }
    \label{fig4}
    \end{figure}

\subsection{Fourier Sine/Cosine Series}

Finally, we consider random (truncated) Fourier sine/cosine series:
\[
    f_j(z) = 
    \left\{
    \begin{array}{ll}
        \cos(\frac{j}{2} z), & \qquad \text{$j$ even}, \\[1ex]
        \sin(\frac{j+1}{2} z), & \qquad \text{$j$ odd}. 
    \end{array}
    \right.
\]
    \begin{figure}[t]
    \begin{center}
    \includegraphics[width=5.0in]{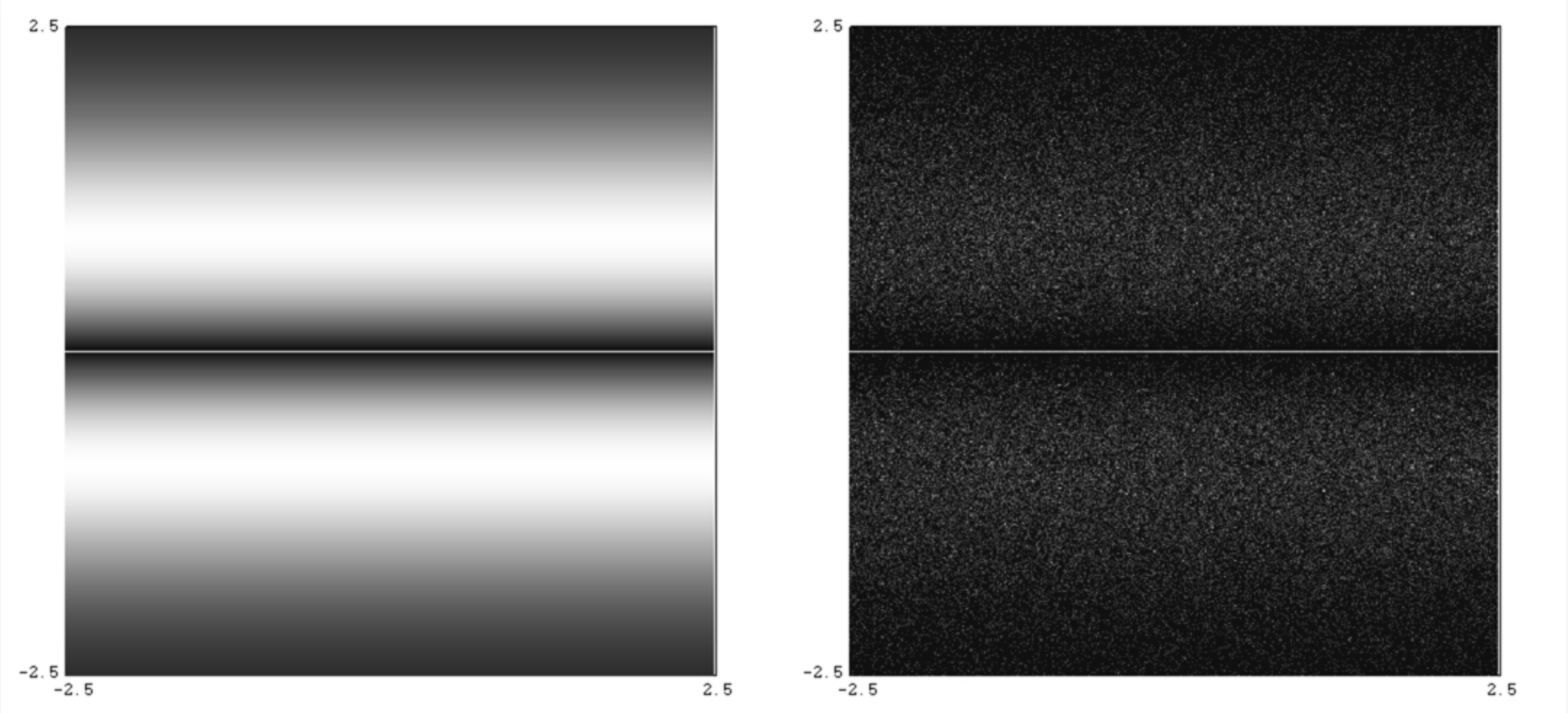}
    \end{center}
    \caption{Random sum of first $3$ terms in a sine/cosine Fourier series:
	    $\eta_0 + \eta_1 \sin(z) + \eta_2 \cos(z)$.
	    The empirical distribution on the right was generated using
	    $210{,}000$ random sums.
    }
    \label{fig5}
    \end{figure}
    \begin{figure}[t]
    \begin{center}
    \includegraphics[width=5.0in]{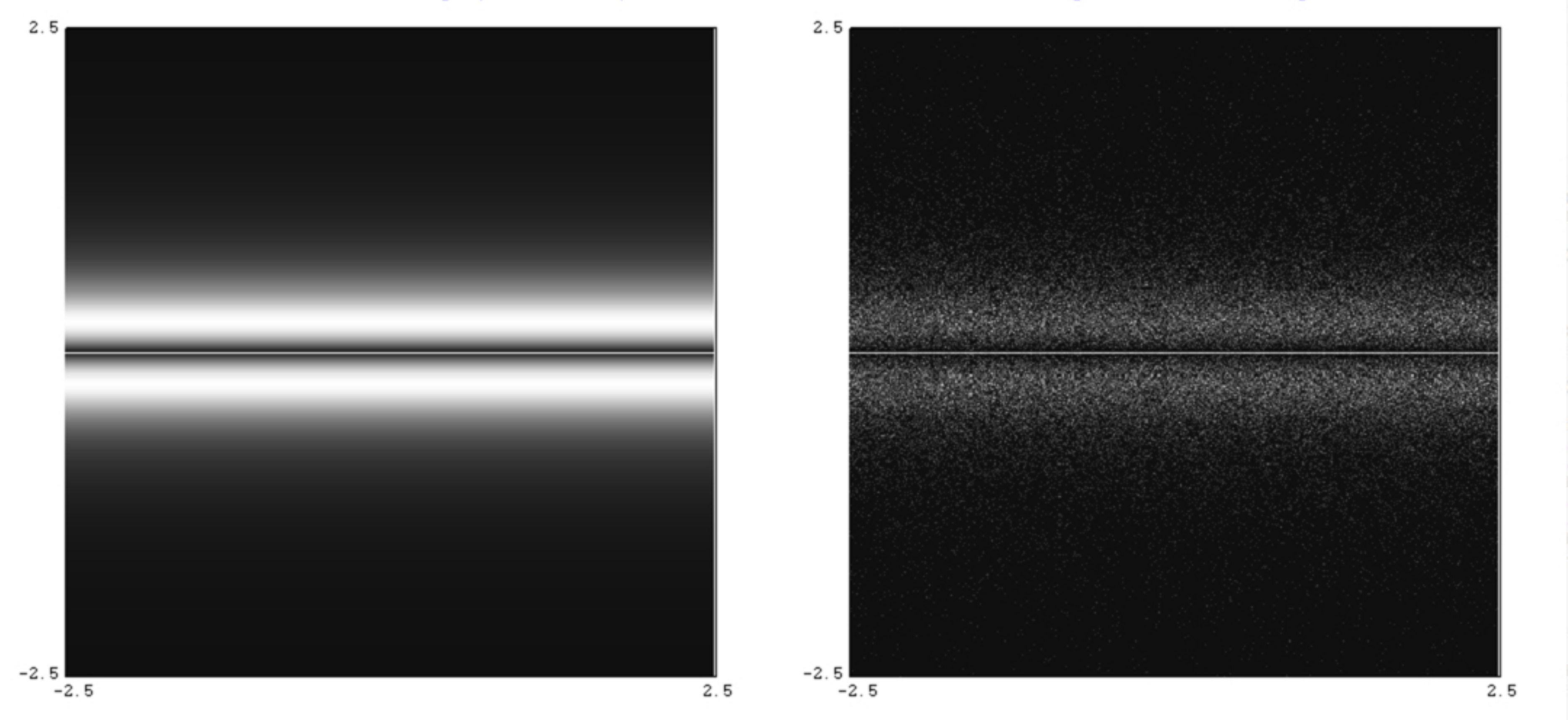}
    \end{center}
    \caption{Random sum of first $11$ terms in a sine/cosine Fourier series:
	    $\eta_0 + \eta_1 \sin(z) + \eta_2 \cos(z)
	    + \cdots + \eta_{9} \sin(5 z) + \eta_{10} \cos(5 z)$.
	    The empirical distribution on the right was generated using
	    $26{,}000$ random sums.
    }
    \label{fig6}
    \end{figure}
The $n=2$ case is shown in Figure \ref{fig5} and the $n=10$ case is shown in
Figure \ref{fig6}.
For this example, it is easy to compute the key functions.  Assuming that $n$
is even and letting $m = n/2$, we get
\[
    \begin{array}{rclrcl}
        A_0(z) & = & m+1 & \quad
        B_0(z) & = & m+1 + 2 \ds\sum_{j=1}^m \sinh^2(jy) \\[1ex]
        A_1(z) & = & 0 & \quad
        B_1(z) & = & -2 i \ds\sum_{j=1}^m j \cosh(jy) \sinh(jy) \\[1ex]
        A_2(z) & = & m (m+1) (2m+1) / 6 & \quad
        B_2(z) & = & 2 \ds\sum_{j=1}^m j^2 \sinh^2(jy) 
    \end{array}
\]
From these explicit formulas, it is easy to check that the density function
$h_n(z)$ depends only on the imaginary part of $z$ as in evident in Figures
\ref{fig5} and \ref{fig6}.
It is also easy to check that the distribution on the real axis is uniform;
i.e., the density function $g_n(x)$ is a constant:
\[
    g_n(x) = \frac{1}{2\pi} \sqrt{n(n+1)/3} .
\]

\section{Final Comments and Suggested Future Research.} \label{sec6}

    The machinery developed in this paper can be applied in many situations
    that we have not covered.  For example, if the coefficients are assumed to
    be independent complex Gaussains (instead of real), then we can apply the
    same methods and we expect that the computations will be simpler.  In this case, the
    intensity function does not have mass concentrated on the real axis (i.e.,
    $g_n = 0$) and the intensity function is rotationally invariant.  

\section*{Acknowledgement}
The author would like to thank John P. D'Angelo for helpful comments.

\bibliographystyle{imsart-ims/imsart-number}
\bibliography{../lib/refs,rootrefs}

\begin{thebibliography}{25}

\bibitem{Ahl66}
\begin{bbook}[author]
\bauthor{\bsnm{Ahlfors},~\bfnm{L.~V.}\binits{L.~V.}}
(\byear{1966}).
\btitle{Complex Analysis}.
\bpublisher{McGraw-Hill, New York}.
\end{bbook}
\endbibitem

\bibitem{BRS86}
\begin{bbook}[author]
\bauthor{\bsnm{Bharucha-Reid},~\bfnm{A.~T.}\binits{A.~T.}} \AND
  \bauthor{\bsnm{Sambandham},~\bfnm{M.}\binits{M.}}
(\byear{1986}).
\btitle{Random Polynomials}.
\bpublisher{Academic Press}.
\end{bbook}
\endbibitem

\bibitem{DPSZ02}
\begin{barticle}[author]
\bauthor{\bsnm{Dembo},~\bfnm{Amir}\binits{A.}},
  \bauthor{\bsnm{Poonen},~\bfnm{Bjorn}\binits{B.}},
  \bauthor{\bsnm{Shao},~\bfnm{Qi-Man}\binits{Q.-M.}} \AND
  \bauthor{\bsnm{Zeitouni},~\bfnm{Ofer}\binits{O.}}
(\byear{2002}).
\btitle{Random polynomials having few or no real zeros}.
\bjournal{Journal of the American Mathematical Society}
\bvolume{15}
\bpages{857--892}.
\end{barticle}
\endbibitem

\bibitem{EK95}
\begin{barticle}[author]
\bauthor{\bsnm{Edelman},~\bfnm{Alan}\binits{A.}} \AND
  \bauthor{\bsnm{Kostlan},~\bfnm{Eric}\binits{E.}}
(\byear{1995}).
\btitle{How many zeros of a random polynomial are real?}
\bjournal{Bulletin of the American Mathematical Society}
\bvolume{32}
\bpages{1--37}.
\end{barticle}
\endbibitem

\bibitem{ET50}
\begin{barticle}[author]
\bauthor{\bsnm{Erd\"{o}s},~\bfnm{P.}\binits{P.}} \AND
  \bauthor{\bsnm{Tur\'{a}n},~\bfnm{P.}\binits{P.}}
(\byear{1950}).
\btitle{On the distribution of roots of polynomials}.
\bjournal{Ann. Math.}
\bvolume{51}
\bpages{105-119}.
\end{barticle}
\endbibitem

\bibitem{Fel12}
\begin{barticle}[author]
\bauthor{\bsnm{Feldheim},~\bfnm{N.~D.}\binits{N.~D.}}
(\byear{2012}).
\btitle{Zeros of Gaussian Analytic Functions with Tranlation-Invariant
  Distribution}.
\bjournal{Israel J. of Math.}
\bvolume{195}
\bpages{317--345}.
\end{barticle}
\endbibitem

\bibitem{FH99}
\begin{barticle}[author]
\bauthor{\bsnm{Forrester},~\bfnm{PJ}\binits{P.}} \AND
  \bauthor{\bsnm{Honner},~\bfnm{G}\binits{G.}}
(\byear{1999}).
\btitle{Exact statistical properties of the zeros of complex random
  polynomials}.
\bjournal{Journal of Physics A: Mathematical and General}
\bvolume{32}
\bpages{2961}.
\end{barticle}
\endbibitem

\bibitem{Ham56}
\begin{binproceedings}[author]
\bauthor{\bsnm{Hammersley},~\bfnm{J.}\binits{J.}}
(\byear{1956}).
\btitle{The zeros of a random polynomial}.
In \bbooktitle{Proc. Third Berkeley Symp. Math. Stat. Probability}
\bvolume{2}
\bpages{89-111}.
\end{binproceedings}
\endbibitem

\bibitem{HN08}
\begin{barticle}[author]
\bauthor{\bsnm{Hughes},~\bfnm{Christopher~P}\binits{C.~P.}} \AND
  \bauthor{\bsnm{Nikeghbali},~\bfnm{A}\binits{A.}}
(\byear{2008}).
\btitle{The zeros of random polynomials cluster uniformly near the unit
  circle}.
\bjournal{Compositio Mathematica}
\bvolume{144}
\bpages{734--746}.
\end{barticle}
\endbibitem

\bibitem{IZ13}
\begin{bincollection}[author]
\bauthor{\bsnm{Ibragimov},~\bfnm{Ildar}\binits{I.}} \AND
  \bauthor{\bsnm{Zaporozhets},~\bfnm{Dmitry}\binits{D.}}
(\byear{2013}).
\btitle{On distribution of zeros of random polynomials in complex plane}.
In \bbooktitle{Prokhorov and contemporary probability theory}
\bpages{303--323}.
\bpublisher{Springer}.
\end{bincollection}
\endbibitem

\bibitem{IZ97}
\begin{barticle}[author]
\bauthor{\bsnm{Ibragimov},~\bfnm{Ildar}\binits{I.}} \AND
  \bauthor{\bsnm{Zeitouni},~\bfnm{Ofer}\binits{O.}}
(\byear{1997}).
\btitle{On roots of random polynomials}.
\bjournal{Transactions of the American Mathematical Society}
\bvolume{349}
\bpages{2427--2441}.
\end{barticle}
\endbibitem

\bibitem{Wil97}
\begin{barticle}[author]
\bauthor{\bsnm{Jr.},~\bfnm{J.~E.~Wilkins}\binits{J.~E.~W.}}
(\byear{1997}).
\btitle{The Expected Value of the Number of Real Zeros of a Random Sum of
  Legendre Polynomials}.
\bjournal{Proc. Amer. Math. Soc.}
\bvolume{125}
\bpages{1531--1536}.
\end{barticle}
\endbibitem

\bibitem{KZ13}
\begin{barticle}[author]
\bauthor{\bsnm{Kabluchko},~\bfnm{Z.}\binits{Z.}} \AND
  \bauthor{\bsnm{Zaporozhets},~\bfnm{D.}\binits{D.}}
(\byear{13}).
\btitle{Roots of Random Polynomials whose Coefficients Have Logarithmic Tails}.
\bjournal{Ann. Prob.}
\bvolume{41}
\bpages{3542--3581}.
\end{barticle}
\endbibitem

\bibitem{KZ14}
\begin{barticle}[author]
\bauthor{\bsnm{Kabluchko},~\bfnm{Z.}\binits{Z.}} \AND
  \bauthor{\bsnm{Zaporozhets},~\bfnm{D.}\binits{D.}}
(\byear{14}).
\btitle{Asymptotic Distribution of Complex Zeros of Random Analytic Functions}.
\bjournal{Ann. Prob.}
\bvolume{42}
\bpages{1374--1395}.
\end{barticle}
\endbibitem

\bibitem{Kac43}
\begin{barticle}[author]
\bauthor{\bsnm{Kac},~\bfnm{M.}\binits{M.}}
(\byear{1943}).
\btitle{On the average number of real roots of a random algebraic equation}.
\bjournal{Bull. Amer. Math. Soc.}
\bvolume{49}
\bpages{314-320,938}.
\end{barticle}
\endbibitem

\bibitem{Kac59}
\begin{bbook}[author]
\bauthor{\bsnm{Kac},~\bfnm{M.}\binits{M.}}
(\byear{1959}).
\btitle{Probability and related topics in physical sciences}.
\bpublisher{Interscience, London}.
\end{bbook}
\endbibitem

\bibitem{LW09}
\begin{barticle}[author]
\bauthor{\bsnm{Li},~\bfnm{Wenbo}\binits{W.}} \AND
  \bauthor{\bsnm{Wei},~\bfnm{Ang}\binits{A.}}
(\byear{2009}).
\btitle{On the expected number of zeros of a random harmonic polynomial}.
\bjournal{Proceedings of the American Mathematical Society}
\bvolume{137}
\bpages{195--204}.
\end{barticle}
\endbibitem

\bibitem{Pri17}
\begin{barticle}[author]
\bauthor{\bsnm{Pritsker},~\bfnm{I.~E.}\binits{I.~E.}}
(\byear{2017}).
\btitle{Zero Distribution of Random Polynomials}.
\bjournal{J. Anal. Math.}
\bnote{To appear}.
\end{barticle}
\endbibitem

\bibitem{PY15}
\begin{barticle}[author]
\bauthor{\bsnm{Pritsker},~\bfnm{I.~E.}\binits{I.~E.}} \AND
  \bauthor{\bsnm{Yeager},~\bfnm{M.~A.}\binits{M.~A.}}
(\byear{2015}).
\btitle{Zeros of Polynomials with Random Coefficients}.
\bjournal{J. Approx. Theory}
\bvolume{189}
\bpages{88--100}.
\end{barticle}
\endbibitem

\bibitem{Pro96}
\begin{barticle}[author]
\bauthor{\bsnm{Prosen},~\bfnm{Tomaz}\binits{T.}}
(\byear{1996}).
\btitle{Exact statistics of complex zeros for Gaussian random polynomials with
  real coefficients}.
\bjournal{Journal of Physics A: Mathematical and General}
\bvolume{29}
\bpages{4417}.
\end{barticle}
\endbibitem

\bibitem{SM09}
\begin{barticle}[author]
\bauthor{\bsnm{Schehr},~\bfnm{G.}\binits{G.}} \AND
  \bauthor{\bsnm{Majumdar},~\bfnm{S.~N.}\binits{S.~N.}}
(\byear{2009}).
\btitle{Condensation of the Roots of Real Random Polynomials on the Real Axis}.
\bjournal{J. Stat. Physics}
\bvolume{135}
\bpages{587--598}.
\end{barticle}
\endbibitem

\bibitem{Van94d}
\begin{barticle}[author]
\bauthor{\bsnm{Shepp},~\bfnm{L.~A.}\binits{L.~A.}} \AND
  \bauthor{\bsnm{Vanderbei},~\bfnm{R.~J.}\binits{R.~J.}}
(\byear{1995}).
\btitle{The complex zeros of random polynomials}.
\bjournal{Transactions of the AMS}
\bvolume{347}
\bpages{4365-4384}.
\end{barticle}
\endbibitem

\bibitem{SZ03}
\begin{barticle}[author]
\bauthor{\bsnm{Shiffman},~\bfnm{Bernard}\binits{B.}} \AND
  \bauthor{\bsnm{Zelditch},~\bfnm{Steve}\binits{S.}}
(\byear{2003}).
\btitle{Equilibrium distribution of zeros of random polynomials}.
\bjournal{International Mathematics Research Notices}
\bvolume{2003}
\bpages{25--49}.
\end{barticle}
\endbibitem

\bibitem{TV14}
\begin{barticle}[author]
\bauthor{\bsnm{Tao},~\bfnm{Terence}\binits{T.}} \AND
  \bauthor{\bsnm{Vu},~\bfnm{Van}\binits{V.}}
(\byear{2014}).
\btitle{Local universality of zeroes of random polynomials}.
\bjournal{International Mathematics Research Notices}
\bpages{1--84}.
\end{barticle}
\endbibitem

\bibitem{SS62}
\begin{barticle}[author]
\bauthor{\bsnm{\v{S}paro},~\bfnm{D.~I.}\binits{D.~I.}} \AND
  \bauthor{\bsnm{\v{S}ur},~\bfnm{M.~G.}\binits{M.~G.}}
(\byear{1962}).
\btitle{On the distribution of roots of random polynomials}.
\bjournal{Vestn. Mosk. Univ., Ser. 1: Mat., Mekh.}
\bpages{40-53}.
\end{barticle}
\endbibitem

\end{thebibliography}

\appendix
\section{Algebraic Simplification of the Formula for 
	$\partial F / \partial \bar{z}$.}

Substituting the derivatives given in \eqref{201} and \eqref{202} into the
formula \eqref{200} for $\partial F / \partial \bar{z}$, we get that the
denominator simplifies to
\begin{eqnarray*}
	\mbox{denom} \left(\frac{\partial F}{\partial \bar{z}}\right)
	& = & \left( B_0 D_0 + B_0^2 - \bar{A}_0 A_0 \right)^2 \\
	& = & \left( B_0 D_0 + D_0^2 \right)^2 \\[1ex]
	& = & (B_0 + D_0)^2 D_0^2 .
\end{eqnarray*}
and the numerator of the formula becomes
\begin{eqnarray*} 
	\mbox{num} \left(\frac{\partial F}{\partial \bar{z}}\right)
	& = &
	    (B_0 D_0 + B_0^2 - \bar{A}_0 A_0)
	    \Bigg(
		B_2 D_0 + B_1 
	       \frac{B_0 \bar{B}_1 - A_0 \bar{A}_1}{D_0}
	       \\ && \qquad \qquad \qquad \qquad
	       + \bar{B}_1 B_1
	       + B_0 B_2 - 2 \bar{A}_1 A_1
	     \Bigg)
	\\
	& &
	    - 
	    (B_1 D_0 + B_0 B_1 - \bar{A}_0 A_1)
	    \Bigg(
		\bar{B}_1 D_0 + B_0 
	       \frac{B_0 \bar{B}_1 - A_0 \bar{A}_1}{D_0}
	       \\ && \qquad \qquad \qquad \qquad
	       + 2 B_0 \bar{B}_1 - 2 \bar{A}_1 A_0
	     \Bigg) .
\end{eqnarray*}
The first step to simplifying the numerator is to replace $B_0^2 - \bar{A}_0
A_0$ in the first term with $D_0$ (like we did in the denominator) and factor
out a $1/D_0$ to get
\begin{eqnarray*} 
	\mbox{num} \left(\frac{\partial F}{\partial \bar{z}}\right)
	& = &
	    \frac{1}{D_0} 
	    (B_0 + D_0) D_0
	    \Big(
		B_2 D_0^2 + B_1 
	             B_0 \bar{B}_1 - B_1 A_0 \bar{A}_1
	       \\ && \qquad \qquad \qquad \qquad
	       + \bar{B}_1 B_1 D_0
	       + B_0 B_2 D_0 - 2 \bar{A}_1 A_1 D_0
	     \Big)
	\\
	& &
	    - \frac{1}{D_0} 
	    ((B_0 + D_0) B_1 - \bar{A}_0 A_1)
	    \Big(
		\bar{B}_1 D_0^2 + B_0 
	             B_0 \bar{B}_1 - B_1 A_0 \bar{A}_1
	       \\ && \qquad \qquad \qquad \qquad
	       + 2 B_0 \bar{B}_1 D_0 - 2 \bar{A}_1 A_0 D_0
	     \Big) .
\end{eqnarray*}
Next, we bundle together the terms that have a $B_0 + D_0$ factor:
\begin{eqnarray*} 
	\mbox{num} \left(\frac{\partial F}{\partial \bar{z}}\right)
	& = &
	    \frac{1}{D_0} 
	    (B_0 + D_0) 
	    \Big(
		B_2 D_0^3 + B_1 B_0 \bar{B}_1 D_0 - B_1 A_0 \bar{A}_1 D_0
	       \\ && \qquad \qquad \qquad \qquad
	       + \bar{B}_1 B_1 D_0^2
	       + B_0 B_2 D_0^2 - 2 \bar{A}_1 A_1 D_0^2
	\\[1ex]
	& &
	    \qquad \qquad \qquad
		- |B_1|^2 D_0^2 - B_0 B_0 |B_1|^2 + B_1^2 A_0 \bar{A}_1
	       \\[1ex] && \qquad \qquad \qquad \qquad
	        - 2 B_0 |B_1|^2 D_0 + 2 \bar{A}_1 A_0 B_1 D_0
	     \Big) \\
	& &
	    + \frac{1}{D_0} 
	    \bar{A}_0 A_1
	    \Big(
		\bar{B}_1 D_0^2 + B_0 
	             B_0 \bar{B}_1 - B_1 A_0 \bar{A}_1
	       \\ && \qquad \qquad \qquad \qquad
	       + 2 B_0 \bar{B}_1 D_0 - 2 \bar{A}_1 A_0 D_0
	     \Big) .
\end{eqnarray*}
Now, there are several places where we can find $B_0 + D_0$ factors.
For example, the big factor containing eleven terms can be rewritten as follows:
\begin{eqnarray*} 
 	&&
	B_2 D_0^3 + B_1 B_0 \bar{B}_1 D_0 - B_1 A_0 \bar{A}_1 D_0
       + \bar{B}_1 B_1 D_0^2
       + B_0 B_2 D_0^2 - 2 \bar{A}_1 A_1 D_0^2 
       \\[1ex]
 	&&
	- |B_1|^2 D_0^2 - B_0 B_0 |B_1|^2 + B_1^2 A_0 \bar{A}_1
	- 2 B_0 |B_1|^2 D_0 + 2 \bar{A}_1 A_0 B_1 D_0 
	\\[1ex]
 	&&
       \qquad 
       = \quad
       (B_0 + D_0) (B_2 D_0^2 - B_0 |B_1|^2 + A_0 \bar{A}_1 B_1) 
	- 2 |A_1|^2 D_0^2 .
\end{eqnarray*}
We also look for $B_0 + D_0$ factors in the five-term factor:
\begin{eqnarray*}
 	&&
	\bar{B}_1 D_0^2 + B_0 
	     B_0 \bar{B}_1 - B_1 A_0 \bar{A}_1
       + 2 B_0 \bar{B}_1 D_0 - 2 \bar{A}_1 A_0 D_0
       \\
 	&&
	\qquad = \;
	(B_0 + D_0) 
	\left( (B_0 + D_0) \bar{B}_1 - A_0 \bar{A}_1 \right) 
	- A_0 \bar{A}_1 D_0 .
\end{eqnarray*}
Substituting these expressions into the formula for the numerator, we get
\begin{eqnarray*} 
	D_0 \; \mbox{num} \left(\frac{\partial F}{\partial \bar{z}}\right)
	& = &
	    (B_0 + D_0) 
	    \left(
	       (B_0 + D_0) (B_2 D_0^2 - B_0 |B_1|^2 + A_0 \bar{A}_1 B_1) 
		- 2 |A_1|^2 D_0^2
	     \right) \\
	& &
	    + 
	    \bar{A}_0 A_1
	    \left(
		(B_0 + D_0) 
		\left( (B_0 + D_0) \bar{B}_1 - A_0 \bar{A}_1 \right) 
		- A_0 \bar{A}_1 D_0
	     \right) .
\end{eqnarray*}
Rearranging the terms, we see that
\begin{eqnarray*} 
	D_0 \; \mbox{num} \left(\frac{\partial F}{\partial \bar{z}}\right)
	& = &
	    (B_0 + D_0)^2 
	       (B_2 D_0^2 - B_0 |B_1|^2 + A_0 \bar{A}_1 B_1 + \bar{A}_0 A_1 \bar{B}_1) 
	\\
	& &
	    - (B_0 + D_0) \left( 2 |A_1|^2 D_0^2 + |A_0|^2 |A_1|^2 \right)
	\\[1ex]
	& &
	    - |A_0|^2 |A_1|^2 D_0 .
\end{eqnarray*}
Here's the tricky part... replace $D_0^2$ in the second row with 
$B_0^2 - |A_0|^2$ and the second and third row simplify nicely:
\begin{eqnarray*} 
	&&
        (B_0 + D_0) \left( 2 |A_1|^2 D_0^2 + |A_0|^2 |A_1|^2 \right)
        + |A_0|^2 |A_1|^2 D_0 
        \\
	&& \qquad \qquad =
        (B_0 + D_0) \left( 2 |A_1|^2 B_0^2 - |A_0|^2 |A_1|^2 \right)
        + |A_0|^2 |A_1|^2 D_0 
        \\
	&& \qquad \qquad =
        2 |A_1|^2 B_0^3 - |A_0|^2 |A_1|^2 B_0 
        + 2 |A_1|^2 B_0^2 D_0 
        \\
	&& \qquad \qquad =
        |A_1|^2 B_0 \left( 2 B_0^2 - |A_0|^2 + 2 B_0 D_0 \right)
        \\
	&& \qquad \qquad =
        |A_1|^2 B_0 \left( B_0 + D_0 \right)^2 .
\end{eqnarray*}
Substituting this expression into our formula for the numerator, we now have
\begin{eqnarray*} 
	D_0 \; \mbox{num} \left(\frac{\partial F}{\partial \bar{z}}\right)
	& = &
	    (B_0 + D_0)^2 
	    (B_2 D_0^2 - B_0 (|A_1|^2 + |B_1|^2) + A_0 \bar{A}_1 B_1 + \bar{A}_0 A_1 \bar{B}_1) .
\end{eqnarray*}
Finally, we get a simple formula for $\partial F / \partial \bar{z}$:
\[
    \frac{\partial F }{ \partial \bar{z} }
    =
    \frac{B_2 D_0^2 - B_0 (|A_1|^2 + |B_1|^2) + A_0 \bar{A}_1 B_1 + \bar{A}_0
		    A_1 \bar{B}_1}{D_0^3} .
\]

\end{document}